\documentclass[12pt]{amsart}
\usepackage{extsizes}
\usepackage{fullpage}
\usepackage{amsfonts}
\usepackage{amssymb}
\usepackage[foot]{amsaddr}
\newtheorem{theorem}{Theorem}[section]
\newtheorem{lemma}[theorem]{Lemma}
\newtheorem{proposition}[theorem]{Proposition}
\newtheorem{corollary}[theorem]{Corollary}
\newtheorem{definition}[theorem]{Definition}
\newtheorem{remark}[theorem]{Remark}

\renewcommand{\pmod}[1]{{\ifmmode\text{\rm\ (mod~$#1$)}\else\discretionary{}{}{\hbox{ }}\rm(mod~$#1$)\fi}}
\begin{document}
\title{Special Values of Anticyclotomic $L$-functions Modulo $\lambda$}
\author{Alia Hamieh}
\address{Department of Mathematics and Statistics, Queen's University, 516 Jeffery Hall, University Avenue, Kingston, ON K7L3N6, Canada}
 \email{ahamieh@mast.queensu.ca}
\begin{abstract}
The purpose of this article is to generalize some results of Vatsal on the special values of Rankin-Selberg L-functions in an anticyclotomic $\mathbb{Z}_{p}$-extension. Let $g$ be a cuspidal Hilbert modular newform of parallel weight $(2,...,2)$ and level $\mathcal{N}$ over a totally real field $F$, and let $K/F$ be a totally imaginary quadratic extension of relative discriminant $\mathcal{D}$. We study the $l$-adic valuation of the special values $L(g,\chi,\frac{1}{2})$ as $\chi$ varies over the ring class characters of $K$ of $\mathcal{P}$-power conductor, for some fixed prime ideal $\mathcal{P}$. We prove our results under the only assumption that the prime to $\mathcal{P}$ part of $\mathcal{N}$ is relatively prime to $\mathcal{D}$. \end{abstract}
\maketitle
\thispagestyle{empty}
\setcounter{section}{-1}
\section{Introduction}

Let $E$ be an elliptic curve over $\mathbb{Q}$ of conductor $N$, and let $K/\mathbb{Q}$ be an imaginary quadratic field extension of discriminant $D$ such that $N$ and $D$ are relatively prime. Denote by $K_{\infty}$ the anticyclotomic $\mathbb{Z}_{p}$-extension of $K$, where $p$ is a given prime number with $p\nmid ND$. In 2002, Vatsal succeeded in settling a conjecture of Mazur pertaining to the size of the Mordell-Weil group $E(K_{\infty})$. In fact, Mazur's conjecture predicts that the group $E(K_{\infty})$ is finitely generated, and Vatsal proved in \cite{vatsal1} that this is true, at least when $E$ is ordinary at $p$, or when the class number of $K$ is prime to $p$.

 In more concrete terms, Vatsal considered the modular form $g$ associated to $E$ and the family of Rankin-Selberg $L$-functions $L(g,\chi,s)$ as $\chi$ varies over ring class characters of $K$ of $p$-power conductor. Under certain conditions on $g$ and $\chi$, the result of Vatsal asserts that the special values $L(g,\chi,1)$ are non-vanishing for all but finitely many $\chi$, provided that $p$ is an ordinary prime for $g$ or $p$ does not divide the class number of $K$. One consequence of this result is the non-triviality of certain Euler systems as formulated by Bertolini-Darmon in \cite{ber-dar} which in its turn implies that the desired statement about the Mordell-Weil group is true.  
 
In 2004, Cornut and Vatsal generalized in \cite{cor-vat} the above mentioned work of Vatsal to totally real fields.  Numerous technical complications arise due to the fact that a more general number field $F$ is considered. However, the basic arguments are ultimately the same, as the authors invoke deep theorems of Ratner \cite{ratner} on uniform distribution of unipotent orbits on $p$-adic Lie groups to deduce the desired result.
 
In 2003, Vatsal extended the results and methods of \cite{vatsal1} to study the variation of the $\lambda$-adic absolute value of of the algebraic part of $L(g,\chi,1)$ as a function of $\chi$, where $\lambda$ is a fixed prime of $\bar{\mathbb{Q}}$ with residue characteristic $l$.  The object of this paper is to generalize this work to totally real fields while removing most of the restrictions on $N$, $p$, $D$ and $l$ $($Theorem \ref{mainthm}$)$. We use the improved formalism developed in \cite{cor-vat} to achieve this purpose. 

We now give a brief account of the results in this work. Let $\pi$ be an irreducible automorphic representation of $GL_{2}$ over a totally real field $F$ corresponding to a cuspidal Hilbert modular newform $g$ of level $\mathcal{N}$, trivial character and parallel weight $(2,...,2)$. The Hecke eigenvalues of $g$ are denoted by $a_{v}$ $(T_{v}g=a_{v}g)$.  Let $K$ be a totally imaginary quadratic extension of $F$ with discriminant $\mathcal{D}$. We fix a prime ideal $\mathcal{P}$ of $F$ such that $\mathcal{P}$ lies over an odd rational prime $p$ with inertia degree one and consider ring class characters of $K$ of $\mathcal{P}$-power conductor which are trivial when restricted to $\mathbb{A}^{*}_{F}\subset\mathbb{A}^{*}_{K}$. We assume that the prime to $\mathcal{P}$ part of $\mathcal{N}$ is relatively prime to $\mathcal{D}$. We also impose sufficient conditions to make the sign in the functional equation of $L(\pi,\chi,s)$ equal $+1$ for all but finitely many characters $\chi$ of the type considered above.

Adapting the notation from \cite{cor-vat}, let $G(n)$ be the Galois group of the ring class field of conductor $\mathcal{P}^{n}$ over $K$, and let $G_{0}$ be the torsion subgroup of $G(\infty)=\varprojlim G(n)$. It is shown in \cite{cor-vat} that $G_{0}$ is finite and can be identified with its image $G_{0}(n)\subset G(n)$ if $n\gg0$, in which case we denote the quotient group $G(n)/G_{0}(n)$ by $H(n)$. The reciprocity map of $K$ maps $\mathbb{A}^{*}_{F}\subset\mathbb{A}^{*}_{K}$ onto a subgroup $G_{2}\simeq Pic(O_{F})$ of $G_{0}$. Using this reciprocity map, we shall identify ring class characters of $K$ of conductor $\mathcal{P}^{n}$ which are trivial on $\mathbb{A}^{*}_{F}$ with primitive characters of $G(n)$ which are trivial on $G_{2}$. Given a character $\chi_{0}$ of $G_{0}$ which is trivial on $G_{2}$ and a sufficiently large $n$, we consider all primitive characters $\chi$ of $G(n)$ of the form $\chi=\chi'_{0}\chi_{1}$, where $\chi'_{0}$ is a character of $G(n)$ inducing $\chi_{0}$ on $G_{0}(n)\simeq G_{0}$, and $\chi_{1}$ is a primitive character of $H(n)$ which we also require to be faithful (injective). We denote the set of all such characters by $P(n,\chi_{0})$ . 

It can be shown that $G(n)$ acts simply transitively on the set of CM points of conductor $\mathcal{P}^{n}$ on the Shimura curve associated to some carefully chosen totally definite quaternion algebra $B$.  Let $l$ be a rational prime, and denote by $E_{l}$ the $l$-adically complete discrete valuation ring containing the Hecke eigenvalues of $g$, $\lambda$ its maximal ideal and $E_{\lambda}$ the residue field $E_{l}/\lambda$.  Let us assume for now that $l\neq p$ to simplify the exposition of the introduction. Given a CM point $x$ of conductor $\mathcal{P}^{n}$ and a ring class character $\chi$ of the same conductor, we define the Gross-Zagier sum
$$\text{a}(x,\chi)=\frac{1}{|G_{2}|}\sum_{\sigma\in G(n)}\chi(\sigma)\psi(\sigma.x),$$ where $\psi$ is the $E_{l}$-valued function  introduced in Definition \ref{function}. We note here that $\psi$ is chosen to be unique up to a unit in $E_{l}$. We also note that $\psi$ is not necessarily the newform but rather the Gross-Prasad test vector for the trivial character as we remark in Section $4$. In the light of the existing Gross-Zagier formula $($see Section 4$)$, our job is reduced to studying the $l$-adic valuation of this sum. Before we describe the results we obtained in this direction, we state Vatsal's result in which $\pi$ corresponds to a weight 2 newform $g$ for $\Gamma_{0}(N)$ such that $N$, $p$, and $D$ (the discriminant of the imaginary quadratic extension $K/\mathbb{Q}$) are pairwise relatively prime. Write $N=N^{+}N^{-}$ where $N^{-}$ is divisible only by primes that are inert in $K$. We assume that $N^{-}$ is squarefree and divisible by an odd number of primes. This last assumption guarantees that the twisted $L$-function $L(\pi,\chi,s)$ has a functional equation with sign $+1$ for all anticyclotomic characters $\chi$ of $K$ with $p$-power conductor. 

 The algebraic part of $L(\pi,\chi,\frac{1}{2})$ is defined in \cite{vatsal2} to be the quantity $$L^{\mathrm{al}}(\pi,\chi,\frac{1}{2})=\frac{L(\pi,\chi,\frac{1}{2})}{\Omega_{\pi}^{\mathrm{can}}},$$ where $\Omega_{\pi}^{\mathrm{can}}$ is Hida's canonical period. The following Gross-Zagier formula, due to Zhang, was the point of departure in Vatsal's work: $$|\mathrm{a}(x,\chi)|^{2}=\frac{L^{\mathrm{al}}(\pi,\chi,\frac{1}{2})}{C_{\mathrm{csp}}}\cdotp C_{\chi}u^{2},$$ where $C_{\chi}=\sqrt{D}p^{n}$, $u$ is half the order of $O_{K}^{*}$ and $\mathrm{C}_{\mathrm{csp}}$ is a constant that measures the congruence between $g$ and some cusp forms of lower levels.

We now state Proposition 4.1 in \cite{vatsal2}
\begin{proposition}\label{propvatsal}
Let $\chi_{0}$ be a character of $G_{0}$ such that $\chi_{0}=1$ on $G_{2}$. Then, for all $n\gg0$, there exists a character $\chi\in P(n,\chi_{0})$ such that  \begin{equation}\label{equation0}\mathrm{ord}_{\lambda}(\mathrm{a}(x,\chi))\leq\mu-1,\end{equation} where $\mu$ is the smallest integer such that $a_{q}\not\equiv 1+q\mod\lambda^{\mu}$ for some $q\nmid pND$.\end{proposition}
Let us assume that the order of $\chi_{0}$ is prime to $p$ and that the Hecke field of $g$ is linearly disjoint from the field generated over $\mathbb{Q}$ by the $p^{\text{th}}$ roots of unity. In this case, we know from Theorem 2.11  in \cite{vatsal1} that $L(\pi,\chi,\frac{1}{2})\neq0$ for all $\chi\in P(n,\chi_{0})$ with $n$ sufficiently large. We remark that this statement differs slightly from the statement given in \cite{vatsal1} since the condition on the Hecke field of $g$ was overlooked there. Now let us also assume that $l$ splits completely in the field $\mathbb{Q}(\chi_{0})$ generated by the values of $\chi_{0}$ and that it is inert in the field $\mathbb{Q}(\mu_{p^{\infty}})$ generated by all $p$-power roots of unity. Under these assumptions, Vatsal then observed that the inequality in (\ref{equation0}) holds true for all $\chi\in P(n,\chi_{0})$. Hence, Vatsal arrived at the following theorem as a corollary to Proposition \ref{propvatsal}. This is Corollary 4.2 in \cite{vatsal2} .
\begin{theorem}\label{theoremvatsal}
If the assumptions above hold for $g$, $l$ and the order of $\chi_{0}$, then  \begin{equation}\label{equation1}\mathrm{ord}_{\lambda}\left(\frac{(L^{\mathrm{al}}(\pi,\chi,\frac{1}{2})}{\mathrm{C}_{\mathrm{csp}}}\cdotp C_{\chi}u^{2}\right)\leq{2(\mu-1)}\end{equation}for all $\chi\in P(n,\chi_{0})$, with $n\gg0$.
\end{theorem}
 It is important to remark that the constant $\mu-1$  in the above upper bounds is different than the constant used in \cite{vatsal2}. Moreover, inequalities (\ref{equation0}) and (\ref{equation1}) are mistakenly given as equalities in that article, the source of the mistake being an error made in the proof of Proposition 5.3 part $(2)$. We mention here that establishing the correct statements required us to use $\mu-1$ as opposed to $\mathrm{ord}_{\lambda}(C_{\mathrm{Eis}})$ which appears in Vatsal's work and is described as a constant that essentially measures the congruence between $g$ and the space of Eisenstein series.

Another point worth mentioning here is the following interesting result of Pollack and Weston. Let $\overline{\rho}_{f}$ denote the residual Galois representation associated to $g$ modulo $\lambda$. If we assume that $\overline{\rho}_{f}$ is irreducible and that the pair $(\overline{\rho}_{f},N^{-})$ satisfies hypothesis CR (see \cite{pw} page 2), then we have $$\mathrm{ord}_{\lambda}(L^{\mathrm{al}}(\pi,\chi,\frac{1}{2}) C_{\chi})=\mathrm{ord}_{\lambda}(\mathrm{C}_{\mathrm{csp}})=\sum_{q|N^{-}}t_{g}(q),$$ where the terms in the sum are local Tamagawa components at primes dividing $N^{-}$ (see Theorem 2.3 and Theorem 6.8 in \cite{pw}).

 Our goal in this article is to prove an analogue of Proposition \ref{propvatsal} and Theorem \ref{theoremvatsal} for a Hilbert modular form $g$ over a totally real field $F$, while removing the assumptions on $g$, $l$ and the order of $\chi_{0}$.

We may assume without loss of generality that $E_{l}$ contains the values of $\chi_{0}$ and the $p^{\mathrm{th}}$ roots of unity.
Let $\chi_{0}$ be a character of $G_{0}$ which is trivial on $G_{2}$ and let $\chi$ be any character in $P(n,\chi_{0})$ for $n\gg0$. One can express $\chi$ as $\chi=\chi'_{0}\chi_{1}$, where $\chi'_{0}$ is some character of $G(n)$ inducing $\chi_{0}$ on $G_{0}(n)\simeq G_{0}$, and $\chi_{1}$ is some primitive character of $H(n)$. We consider the trace of $\mathrm{a}(x,\chi)$ taken from $E_{\lambda}(\chi_{1})$ to $E_{\lambda}$: $$\mathrm{Tr}(\mathrm{a}(x,\chi))=\sum_{\sigma\in \mathrm{Gal}(E_{\lambda}(\chi_{1})/E_{\lambda})}\sigma(\mathrm{a}(x,\chi)).$$ This trace expression is different than the average expression $$\mathrm{b}(x,\chi_{0})=\sum_{\chi\in P(n,\chi_{0})}\mathrm{a}(x,\chi)$$ considered in the work of Vatsal and Cornut-Vatsal. In particular, given any $\chi\in P(n,\chi_{0})$, the non-vanishing of $\mathrm{Tr}(\mathrm{a}(x,\chi))$ implies that of $\mathrm{a}(x,\chi)$. After a series of reductions in Section 5, we arrive at the following result in Section 6.

\begin{theorem}
Let $F$ be a totally real number field, and let $K$ be a totally imaginary quadratic extension of $F$ with discriminant $\mathcal{D}$. Let $\pi$ be an automorphic irreducible representation of $GL_{2}$ over $F$ corresponding to an adelic Hilbert modular newform $g$ of level $\mathcal{N}$, trivial central character and parallel weight $(2,...,2)$. We fix a prime ideal $\mathcal{P}$ of $F$ which lies above an odd rational prime $p$ with inertia degree one. We assume that the tuple $(\pi,K,\mathcal{P})$ satisfies hypotheses (1)-(3) in Section 1. Let $l$ be a rational prime $(l\neq p)$. We denote by $E_{l}$ the $l$-adically complete discrete valuation ring containing the Hecke eigenvalues of $g$ and $\lambda$ its maximal ideal.

Fix a character $\chi_{0}$ of $G_{0}$ such that $\chi_{0}=1$ on $G_{2}$. Then, for any $\chi\in P(n,\chi_{0})$ and any CM point $x$ of conductor $\mathcal{P}^{n}$ with $n\gg0$, we have
$$\mathrm{ord}_{\lambda}\left(\mathrm{Tr}(\mathrm{a}(x,\chi))\right)\leq\mu-1,$$ where $\mu$ is precisely given in Section 6 Definition \ref{cst}.
\end{theorem}

\section{Preliminaries and Notations}
Let us first fix some notation. We write $\mathbb{A}$ $($resp. $\mathbb{A}_{f}$$)$ for the ring of adeles $($resp. finite adeles$)$ of $\mathbb{Q}$. Let $F$ be a totally real number field with $[F:\mathbb{Q}]=d$, and let $K$ be a totally imaginary quadratic extension of $F$ the discriminant of which we denote by $\mathcal{D}$. The ring of adeles of $F$ is $\mathbb{A}_{F}=\mathbb{A}\otimes_{\mathbb{Q}}F$, and the ring of finite adeles of $F$ is $\hat{F}=\mathbb{A}_{f}\otimes_{\mathbb{Q}}F$. Similarly, we write $\mathbb{A}_{K}$ $($resp. $\hat{K}$$)$ for the ring of adeles $($resp. finite adeles$)$ of $K$. We denote by $\hat{M}=M\otimes_{\mathbb{Z}}\hat{\mathbb{Z}}$ the profinite completion of a finitely generated $\mathbb{Z}$-module $M$.

We consider an adelic Hilbert modular newform $g$ of level $\mathcal{N}$, trivial central character and parallel weight $(2,...,2)$. 
We denote by $a_{v}$ the Hecke eigenvalues of $g$ $(T_{v}g=a_{v}g)$. Let $\pi$ be the automorphic irreducible representation of $GL_{2}$ over $F$ corresponding to $g$, and denote by $\pi_{v}$  the local component of $\pi$ at  $v$ for any finite place $v$ of $F$.

We fix a prime ideal $\mathcal{P}$ of $F$ which lies above an odd rational prime $p$ with ramification index $e$ and inertia degree 1, so that $|O_{F}/\mathcal{P}|=p$. Let $\varpi_{\mathcal{P}}$ be a uniformizer in $F_{\mathcal{P}}$. We consider finite-order Hecke characters $\chi$ of $K$ of conductor $\mathcal{P}^{n}$ with $n\geq0$.

The data of the previous paragraph is to remain fixed and the following hypotheses are assumed throughout this article:   \begin{enumerate}
\item{The representations $\pi$ and $\pi\otimes\eta$ are distinct, where $\eta$ is the quadratic character associated to the extension $K/F$. We say that the pair $(\pi,K)$ is non-exceptional.}
 \item{The prime to $\mathcal{P}$ part $\mathcal{N}'$ of $\mathcal{N}$ is relatively prime to the discriminant $\mathcal{D}$ of $K/F$}
 \item{Let $S$ be the set of all the Archimedean places of $F$, together with those finite places of $F$ which do not divide $\mathcal{P}$, are inert in $K$, and divide $\mathcal{N}$ to an odd power. We require $S$ to have an even cardinality}
 \item{We let $\chi$ vary through the collection of ring class characters of $\mathcal{P}$-power conductor which satisfy: $\chi$ is trivial when restricted to $\mathbb{A}^{*}_{F}\subset\mathbb{A}^{*}_{K}$.}
 \end{enumerate}
Recall that $L(\pi,\chi,s)$ is the Rankin-Selberg $L$-function associated to $\pi$ and $\pi(\chi)$, where $\pi(\chi)$ is the automorphic representation of $GL_{2}$ attached to $\chi$. It follows from the last condition that the sign in the functional equation of $L(\pi,\chi,s)$ is $+1$ for all but finitely many characters $\chi$ that satisfy condition $(4)$ $($see Lemma $1.1$ in \cite{cor-vat}$)$.

Let $l$ be any rational prime. Fix an embedding $\overline{\mathbb{Q}}\rightarrow\overline{\mathbb{Q}_{l}}$, and denote by $E$ the subalgebra of $\overline{\mathbb{Q}_{l}}$ generated by the images of the Hecke eigenvalues of $g$. Write $E_{l}$ for the integral closure of $E$ in its field of fractions and $\lambda$ for the maximal ideal in $E_{l}$. 

\section{CM Points and Galois Action}

Let $B$ be the totally definite quaternion algebra over $F$ such that Ram$(B)=S$. Let $G=\mathrm{Res}_{F/\mathbb{Q}}(B^{*})$ be the algebraic group over $\mathbb{Q}$ associated to $B^{*}$. Thus, the center of $G$ is $Z=\mathrm{Res}_{F/\mathbb{Q}}(F^{*})$. Since every place in $F$ that ramifies in $B$ is inert in $K$, there exists an $F$-embedding $K\hookrightarrow B$. After fixing such an embedding, the group $T=\mathrm{Res}_{F/\mathbb{Q}}(K^{*})$ can be viewed as a maximal sub-torus of $G$ defined over $\mathbb{Q}$. 

In what follows, we sketch the construction of an $O_{F}$-order $R$ of reduced discriminant $\mathcal{N}$ in $B$ following \cite{cor-vat} and \cite{zhang1}. Let $\mathcal{N}'$ be the prime to $\mathcal{P}$ part of $\mathcal{N}$, and write $\mathcal{N}=\mathcal{P}^{\delta}\mathcal{N}'$. Let $R_{0}$ be an Eichler order of level $\mathcal{P}^{\delta}$ in $B$. We choose $R_{0}$ such that the $O_{F}$-order $O=O_{K}\cap R_{0}$ has a $\mathcal{P}$-power conductor. Denote by $\mathcal{N}_{B}$ the discriminant of $B/F$, and let $\mathcal{M}_{K}$ be an ideal in $O_{K}$ which has relative norm $\mathcal{N}'/\mathcal{N}_{B}$. We may find such an ideal $\mathcal{M}_{K}$ as follows. For each prime $\mathfrak{P}$ dividing $\mathcal{N}'$, let $\mathfrak{P}_{K}$ be a prime of $O_{K}$ dividing $\mathfrak{P}$. If we put  $$\mathcal{M}=\prod_{\mathfrak{P}\mid\mathcal{N}_{B}}\mathfrak{P}_{K}^{[\mathrm{ord}_{\mathfrak{P}}(\mathcal{N})/2]}.\prod_{\mathfrak{P}\mid\frac{\mathcal{N}'}{\mathcal{N}_{B}}}\mathfrak{P}_{K}^{\mathrm{ord}_{\mathfrak{P}}(\mathcal{N})},$$then $$\mathcal{M}_{K}=\prod_{\mathfrak{P}}\mathfrak{P}_{K}^{\mathrm{ord}_{\mathfrak{P}}(\mathcal{M})}.$$ Finally, we obtain $R$ by the following formula: $$R=O+(O\cap\mathcal{M}_{K}).R_{0}.$$

In particular, $R_{\mathcal{P}}=R_{0,\mathcal{P}}$ is an Eichler order of level $\mathcal{P}^{\delta}$ in $B_{\mathcal{P}}\simeq M_{2}(F_{\mathcal{P}})$. Without loss of generality, we make the identification
\begin{equation}\label{order}
R_{\mathcal{P}}=h_{\mathcal{P}}M^{0}(\mathcal{P}^{\delta})h_{\mathcal{P}}^{-1} \text{ for some } h_{\mathcal{P}}\in GL_{2}(F_{\mathcal{P}}),
\end{equation} where $$M^{0}(\mathcal{P}^{\delta})=\left\{\left( \begin{array}{cc} a & b \\ c& d \end{array} \right)\in M_{2}(O_{F_{\mathcal{P}}}):b\equiv0\mod\mathcal{P}^{\delta}\right\}.$$ Hence, we also identify $R^{*}_{\mathcal{P}}$ with the subgroup $h_{\mathcal{P}}K^{0}(\mathcal{P}^{\delta})h_{\mathcal{P}}^{-1}$, where
$$K^{0}(\mathcal{P}^{\delta})=\left\{\left( \begin{array}{cc} a & b \\ c& d \end{array} \right)\in GL_{2}(O_{F_{\mathcal{P}}}):b\equiv0\mod\mathcal{P}^{\delta}\right\}.$$

 Define an open compact subgroup $H$ of $G(\mathbb{A}_{f})$ by $H=\hat{R}^{*}$. The subgroup $H$ is sometimes referred to as the level subgroup. This gives rise to the finite sets $$\mathrm{M}_{H}=G(\mathbb{Q})\backslash G(\mathbb{A}_{f})/H,$$ and $$\mathrm{N}_{H}=Z(\mathbb{Q})^{+}\backslash Z(\mathbb{A}_{f})/\mathrm{nrd}(H).$$ It also gives rise to the set of CM points $$\mathrm{CM}_{H}=T(\mathbb{Q})\backslash G(\mathbb{A}_{f})/H.$$
Notice that any function on $\mathrm{M}_{H}$ induces a function on $\mathrm{CM}_{H}$ via the obvious reduction map $\mathrm{red}:\mathrm{CM}_{H}\rightarrow\mathrm{M}_{H}$.

The action of $T(\mathbb{A}_{f})$ on $\mathrm{CM}_{H}$ by left multiplication in $G(\mathbb{A}_{f})$ factors through the reciprocity map $$\mathrm{ rec}_{K}:T(\mathbb{A}_{f})\rightarrow\mathrm{Gal}_{K}^{\mathrm{ab}}$$ and thus induces an action of Gal$_{K}^{\mathrm{ab}}$ on $\mathrm{CM}_{H}$. Hence, for $x=[g]\in \mathrm{CM}_{H}$ and $\sigma\in$Gal$_{K}^{\mathrm{ab}}$, we have $\sigma.x=[\beta g]$ where $\beta\in T(\mathbb{A}_{f})$ is such that $\mathrm{rec}_{K}(\beta)=\sigma$. 

Moreover, the reduced norm map on $G(\mathbb{A}_{f})$ induces the map $\mathrm{c}:\mathrm{M}_{H}\rightarrow\mathrm{N}_{H}$. Hence, the action of Gal$_{F}^{\mathrm{ab}}$ on $\mathrm{N}_{H}$ induces an action of $\mathrm{Gal}_{K}^{\mathrm{ab}}$ on $\mathrm{N}_{H}$. For $x=[z]\in \mathrm{N}_{H}$ and $\sigma\in$Gal$_{K}^{\mathrm{ab}}$, we have $\sigma.x=[\mathrm{nrd}(\beta) g]$ where $\beta\in T(\mathbb{A}_{f})$ is such that $\mathrm{rec}_{K}(\beta)=\sigma$. 

We now introduce the notion of a CM point with a $\mathcal{P}$-power conductor.

\begin{definition}
We say that $x=[g]\in\mathrm{CM}_{H}$ is a CM point of conductor $\mathcal{P}^{n}$ and write $x\in\mathrm{CM}_{H}(\mathcal{P}^{n})$ if $T(\mathbb{A}_{f})\cap gHg^{-1}=\hat{O}_{\mathcal{P}^{n}}^{*}$, where $O_{\mathcal{P}^{n}}\subset O_{K}$ is the $O_{F}$-order of conductor $\mathcal{P}^{n}$. 
\end{definition}

Choose $\alpha_{\mathcal{P}}\in O_{K_{\mathcal{P}}}$ such that $\{1,\alpha_{\mathcal{P}}\}$ is an $O_{F_{\mathcal{P}}}$-basis of $O_{K_{\mathcal{P}}}$. Since $O_{\mathcal{P}^{n},\mathcal{P}}=O_{F_{\mathcal{P}}}+\mathcal{P}^{n}O_{K_{\mathcal{P}}}$, the set $\{1,\varpi_{\mathcal{P}}^{n}\alpha_{\mathcal{P}}\}$ is an  $O_{F_{\mathcal{P}}}$-basis of $O_{\mathcal{P}^{n},\mathcal{P}}$. We fix the embedding $K_{\mathcal{P}}\hookrightarrow M_{2}(F_{\mathcal{P}})$ defined by $$a+b\alpha_{\mathcal{P}}\mapsto \left( \begin{array}{cc} a+b\mathrm{Tr}\alpha_{\mathcal{P}} & b \\ -b\mathrm{N}\alpha_{\mathcal{P}}& a \end{array} \right),$$where Tr and N denote the trace and norm maps. 

\begin{lemma}\label{cmpoint}
Consider $k_{\mathcal{P}}\in B_{\mathcal{P}}\simeq M_{2}(F_{\mathcal{P}})$ specified as: $$k_{\mathcal{P}}=\left( \begin{array}{cc} \varpi_{\mathcal{P}}^{n-\delta} & 0 \\ 0& 1 \end{array} \right).$$ Then, for $n$ large enough, the order $k_{\mathcal{P}}h_{\mathcal{P}}^{-1}R_{\mathcal{P}}h_{\mathcal{P}}k_{\mathcal{P}}^{-1}$ in $B_{\mathcal{P}}$ optimally contains the $O_{F_{\mathcal{P}}}$-order in $K_{\mathcal{P}}$ of conductor $\mathcal{P}^{n}$.\end{lemma}
\begin{proof}
Let $\tau=a+b\alpha_{\mathcal{P}}$ be any element in $K_{\mathcal{P}}$. We have $$k_{\mathcal{P}}^{-1}\tau k_{\mathcal{P}}=\left( \begin{array}{cc} a+b\mathrm{Tr}\alpha_{\mathcal{P}} & b\varpi_{\mathcal{P}}^{\delta-n} \\ -b\varpi_{\mathcal{P}}^{n-\delta}\mathrm{N}\alpha_{\mathcal{P}}& a \end{array} \right).$$ 

Then, for all $n\gg0$, we have $k_{\mathcal{P}}^{-1}\tau k_{\mathcal{P}}\in M^{0}(\mathcal{P}^{\delta})$ if and only if $b\in\mathcal{P}^{n}$. Hence, the order $k_{\mathcal{P}}h_{\mathcal{P}}^{-1}R_{\mathcal{P}}h_{\mathcal{P}}k_{\mathcal{P}}^{-1}$ in $B_{\mathcal{P}}$ optimally contains the $O_{F_{\mathcal{P}}}$-order in $K_{\mathcal{P}}$ of conductor $\mathcal{P}^{n}$.
\end{proof}
We choose $z=[g_{0}]\in \mathrm{CM}_{H}(\mathcal{P}^{n})$ such that the $\mathcal{P}^{\mathrm{th}}$-component $g_{0,\mathcal{P}}$ of $g_{0}$ is equal to $k_{\mathcal{P}}h_{\mathcal{P}}^{-1}$. This choice of a CM point proves useful in Section 4.

\section{Uniform Distribution of CM Points}
The CM points are uniformly distributed on the components of the Shimura curve associated to $B$. This was the most crucial idea behind Vatsal's proof of Mazur's conjecture for weight two modular forms over $\mathbb{Q}$. In this section, we recall a crucial result on the uniform distribution of CM points due to Cornut-Vatsal, which we use to prove our main theorem. To describe this result, we need to introduce some more notation. 

Let $K[\mathcal{P}^{n}]$ be the ring class field over $K$ of conductor $\mathcal{P}^{n}$. In other words, $K[\mathcal{P}^{n}]$ is the abelian extension of $K$ associated by class field theory to the subgroup $K^{*}K_{\infty}^{*}\hat {O}_{\mathcal{P}^{n}}^{*}$ of $\mathbb{A}_{K}^{*}$, where $K_{\infty}=K\otimes_{\mathbb{Q}}\mathbb{R}$. Let $G(n)$ denote the Galois group of  this extension. We have $$G(n)=\mathrm{Gal}(K[\mathcal{P}^{n}]/K)\simeq \mathbb{A}_{K}^{*}/(K^{*}K_{\infty}^{*}\hat{O}_{\mathcal{P}^{n}}^{*})$$ via the reciprocity map of $K$.

 Set $\displaystyle{K[\mathcal{P}^{\infty}]= \cup_{n\geq0}K[\mathcal{P}^{n}]}$, so that $G(\infty)=\mathrm{Gal}(K[\mathcal{P}^{\infty}]/K)$. The torsion subgroup of $G(\infty)$ is denoted by $G_{0}$. It is finite and $G(\infty)/G_{0}$ is a free $\mathbb{Z}_{p}$-module of rank $[F_{\mathcal{P}}:\mathbb{Q}_{p}]$. The reciprocity map of $K$ maps $\mathbb{A}_{F}^{*}\subset\mathbb{A}_{K}^{*}$ onto the subgroup $G_{2}\simeq Pic(O_{F})$ of $G_{0}.$ 
 
Let $G(\infty)'$ be the subgroup of $G(\infty)$ generated by the Frobeniuses of the primes of $K$ which are not above $\mathcal{P}$. Write $G_{1}=G_{0}\cap G(\infty)'$. Let $\mathcal{D}'$ be the square-free product of the primes $\mathcal{Q}\neq\mathcal{P}$ of $F$ which ramify in $K$. Then $G_{1}/G_{2}$ is an $\mathbb{F}_{2}$-vector space with basis
$$\{\sigma_{\mathcal{Q}} \mod G_{2}: \mathcal{Q}|\mathcal{D}'\},$$ where $\sigma_{\mathcal{Q}}=\text{Frob}_{\mathfrak{Q}}$ and $\mathfrak{Q}$ is the prime of $K$ above $\mathcal{Q}$.

Loosely speaking, the uniform distribution result in \cite{vatsal1} states the following. Let $p_{1}$ and $p_{2}$ be arbitrary double cosets in $\mathrm{M}_{H}$, and let $\sigma$ be an arbitrary nontrivial element of $G_{0}$ with $\sigma\notin G_{1}$. Then there exists a CM point $x\in \mathrm{CM}_{H}(\mathcal{P}^{n})$ such that $\mathrm{red}(x)=p_{1}$ and $\mathrm{red}(\sigma.x)=p_{2}$ whenever $n$ is sufficiently large. 

In what follows, we describe the result of Cornut and Vatsal which extends and refines Vatsal's theorem alluded to in the previous paragraph.

Let $\mathcal{R}$ be a set of representatives for $G_{0}/G_{1}$ containing 1. We have the following maps:
$$\mathrm{RED}:CM_{H}(\mathcal{P}^{\infty})\rightarrow M_{H}^{\mathcal{R}},\text{    }\text{     }\text{     }\text{   } x\mapsto (\mathrm{red}(\tau.x))_{\tau\in\mathcal{R}}$$
$$\mathrm{C}:M_{H}^{\mathcal{R}}\rightarrow N_{H}^{\mathcal{R}},\text{    }\text{     }\text{     }\text{   } \text{   }\text{  }(a_{\tau})_{\tau\in\mathcal{R}}\mapsto (c(a_{\tau}))_{\tau\in\mathcal{R}}$$ 
and the composite map
$$\mathrm{C}\circ \mathrm{RED}: CM_{H}(\mathcal{P}^{\infty})\rightarrow N_{H}^{\mathcal{R}},$$
which is $G(\infty)$-equivariant.

The following is the key theorem of Cornut-Vatsal as stated in \cite{cor-vat}. However, the reader is referred to \cite{cor-vat1} for a proof of this result.
\begin{theorem}[\cite{cor-vat}]\label{cv}
For all but finitely many $x\in \mathrm{CM}_{H}(\mathcal{P}^{\infty})$, $$\mathrm{RED}(G(\infty).x)=\mathrm{C}^{-1}(G(\infty).\text{C}\circ \mathrm{RED}(x))$$
\end{theorem}
\section{A Special Value Formula}

An automorphic form of weight 2 on $G$ is a smooth $($=locally constant$)$ scalar-valued function $\theta$ on $G(\mathbb{Q})\backslash G(\mathbb{A}_{f})$. Adapting the notation from \cite{cor-vat}, we denote the space of all automorphic forms of weight 2 on $G$ by $\mathcal{S}_{2}$. The group $G(\mathbb{A}_{f})$ acts on $\mathcal{S}_{2}$ by right translation:$$(g.\phi)(x)=\phi(xg),\text{         }g\in G(\mathbb{A}_{f})\text{ }\text{and}\text{ }\phi\in \mathcal{S}_{2}.$$
Let $\pi'$ be the unique cuspidal automorphic representation on $B$ that is associated to $\pi$ by the Jacquet-Langlands correspondence $(\pi=\mathrm{JL}(\pi'))$, and let $\mathcal{S}_{2}(\pi')$ be the realization of $\pi'$ in $\mathcal{S}_{2}$. Let 
$$\mathcal{S}_{2}(\pi')^{H}=\{f:\text{M}_{H}\rightarrow E_{l}\}.$$
 We recall here that $H=\hat{R}^{*}$  and $R$ is the quaternion order described in Section 2. In particular, for each finite place $v$ of $F$, the order $R_{v}$ is of reduced discriminant $\mathcal{N}_{v}$ and optimally contains the order $O_{v}$, where $O$ is the order of $K$ of conductor $\mathcal{P}^{\delta}$. It follows from the local theory of newforms for $v$ split in $K$ and the Saito-Tunnel Theorem for $v$ non-split in $K$ that the space $\mathcal{S}_{2}(\pi')^{H}$ has rank one over $E_{l}$. 

\begin{definition}\label{function}
Let $\theta$ be a generator of $\mathcal{S}_{2}(\pi')^{H}$. Define $\psi:\mathrm{CM}_{H}\rightarrow E_{l}$ by $$\psi :=\theta\circ\mathrm{red}.$$
\end{definition}

Notice that by definition $\theta$ is non-zero modulo $\lambda$, and so it is unique up to a unit in $E_{l}$. We also make the following important remark.
\begin{remark}\label{remarknewform}
The function $\theta$ is not necessarily the newform in the usual sense but rather the Gross-Prasad test vector for the trivial character $($see Section $7.1$ in \cite{fmp}$)$. Notice that in Vatsal's setting, the assumptions made on $N^{+}$ and $N^{-}$ $($see page $2$$)$ allow for taking $\theta$ to be the newform. However, in a more general setting, like the one considered in this work, one should take $\theta$ to be the Gross-Prasad test vector.
\end{remark}

Let $\chi$ be a Hecke character of $K$, and consider the linear form $l_{\chi}: \mathcal{S}_{2}(\pi')\rightarrow\mathbb{C}$ defined by the period integral: 
\begin{equation}\label{haar}\displaystyle{ l_{\chi}(\phi)=\int_{Z(\mathbb{A})T(\mathbb{Q})\backslash T(\mathbb{A})} \chi(t)\phi(t)\, dt},\end{equation} where $dt$ is the choice of Haar measure used in \cite{fmp}.
In 1985, Waldspurger proved a fundamental theorem $($Th\'{e}or\`{e}me $2$ in  \cite{waldspurger}$)$ which implies that for a ring class character $\chi$ of $P$-power conductor satisfying $\chi=1$ on $\mathbb{A}_{F}^{*}$, we have
$$ L(\pi,\chi,\frac{1}{2})\neq0\text{     } \Leftrightarrow \text{     }\exists\phi\in\ \mathcal{S}_{2}(\pi'): l_{\chi}(\phi)\neq0.$$
However, the result of Waldspurger doesn't give an explicit formula for the special value $L(\pi,\chi,\frac{1}{2})$. Several authors have subsequently taken up the task of specifying a test vector $\phi$ on which to evaluate the linear functional $l_{\chi}$ and finding a formula for $L(\pi,\chi,\frac{1}{2})$ in terms of $l_{\chi}(\phi)$. Such results have numerous applications to arithmetic and Iwasawa theory. In the rest of this section, we briefly present some of this work in the framework of this article.

 In \cite{zhang}, Zhang has obtained a precise formula for $L(\pi,\chi,\frac{1}{2})$ under the assumption that $\mathcal{N}$, $\mathcal{P}^{n}$ and $\mathcal{D}$ are pairwise co-prime. A stronger result in this direction has been made in \cite{martin-whitehouse} using the relative trace formula of Jacquet and Chen \cite{jc}. Under the only assumption that $\mathcal{N}$ and $\mathcal{P}^{n}$ are co-prime, Martin and Whitehouse established an explicit formula for $L(\pi,\chi,\frac{1}{2})$ in terms of $l_{\chi}(\phi)$, where $\phi$ is a nice test vector extracted from the work of Gross and Prasad \cite{grossprasad}. The authors noted that the restriction on the conductors of $\pi$ and $\chi$ was not essential to their approach but was made only for the sake of the Gross-Prasad test vector. In \cite{fmp}, the authors succeeded in removing this restriction by constructing a suitable test vector when $n\geq1$ and $\mathcal{P}$ divides $\mathcal{N}$ under the following assumption:\begin{equation}\label{inertcondition}\text{if }\mathcal{P}\text{ is inert in }K,\text{ then }n\geq\delta.\end{equation}
 This allowed the authors to obtain a more general version of the main theorem in \cite{martin-whitehouse}. We mention here that although the main formula in \cite{fmp} applies to a more general setting than the one imposed in our present work, we only state it under the hypotheses fixed in Section 1. Before doing so, we need to introduce some notation.

Let $S(\pi)$ (resp. $S(\chi)$) be the set of finite places of $F$ where $\pi$ (resp. $\chi$) is ramified and $S_{2}(\pi)$ the set of places where $c(\pi_{v})\geq2$ ($c(\pi_{v})$ is the exponent of the conductor of $\pi_{v}$). Denote the absolute value of the discriminants of $F$ and $K$ by $\Delta_{F}$ and $\Delta_{K}$. 

\begin{theorem}[\cite{fmp}]\label{martin}
There exists a test vector $\theta_{\chi}\in \mathcal{S}_{2}(\pi')$ such that $$\frac{|l_{\chi}(\theta_{\chi})|^{2}}{(\theta_{\chi},\theta_{\chi})}=\frac{1}{2\pi^{d}p^{n}}\sqrt{\frac{\Delta_{F}}{\Delta_{K}}}L_{S(\chi)}(\eta,1)L_{S(\pi)\cup S(\chi)}(\eta,1)L_{S(\pi)\cap S(\chi)}(1_{F},1)\cdotp\frac{L^{S_{2}(\pi)}(\pi,\chi,\frac{1}{2})}{L^{S_{2}(\pi)}(\pi,Ad,1)}.$$ Here $(.,.)$ is the standard inner product on $\pi'$ with respect to the measure on $G(\mathbb{A})$ which is the product of the local Tamagawa measures multiplied by $L^{S(\pi)}(1_{F},2)$.
\end{theorem}

Our next step is to establish a relation between the test vector $\theta_{\chi}$ and the form $\theta$ which obviously does not depend on $\chi$. In doing so, we produce a CM point $\text{\bf{x}}\in\mathrm{CM}_{H}(\mathcal{P}^{n})$. This step allows us to conveniently express $l_{\chi}(\theta_{\chi})$ in terms of the finite sum: $$\mathrm{a}(\text{\bf{x}},\chi)=\frac{1}{|G_{2}|}\sum_{\sigma\in G(n)}\chi(\sigma)\psi(\sigma.\text{\bf{x}}),$$ which is an essential ingredient in relating questions about the special values $L(\pi,\chi,\frac{1}{2})$ to the distribution of CM points. The basic idea behind this step is extracted from the work of Vatsal in \cite{msri}.

In what follows we briefly describe the local vectors $\theta_{\chi,v}$ (defined in Section 7.1 in \cite{fmp}) which yield the Gross-Prasad test vector $\theta_{\chi}$.

Let $v$ be a finite place of $F$, and set $\mathcal{N}_{v}=(\varpi_{v}^{c(\pi_{v})})$ where $\varpi_{v}$ is any uniformizing parameter in $F_{v}$. If $\pi_{v}$ is an unramified principal series representation, we let $R_{\chi,v}\subset B_{v}$ denote a maximal order which optimally contains the order of $K_{v}$ with the same conductor as $\chi_{v}$; such orders exist and are unique up to $K_{v}^{*}$ conjugacy by Proposition 3.2 in \cite{gross88}. It follows from \cite{grossprasad} that $R^{*}_{\chi,v}$ fixes a unique line in $\pi'_{v}$. We let $\theta_{\chi,v}$ be any nonzero vector which is right invariant under $R^{*}_{\chi,v}$. If $\pi_{v}$ is ramified, $\chi_{v}$ is unramified, and  $K_{v}$ is an unramified field extension of $F_{v}$, we let $R_{\chi,v}$ denote an order of reduced discriminant $\mathcal{N}_{v}$ in $B_{v}$ which contains $O_{K_{v}}$. That such orders exist and are unique up to conjugacy by $K_{v}^{*}$ is guaranteed by Proposition 3.4 in \cite{gross88}. Again, it is shown in \cite{grossprasad} that $R^{*}_{\chi,v}$ fixes a unique line in $\pi'_{v}$, and we let $\theta_{\chi,v}$ be any nonzero vector on this line. If $\pi_{v}$ is ramified, $\chi_{v}$ is unramified and $K_{v}/F_{v}$ is split, the vector $\theta_{\chi,v}$ is chosen as described in \cite{fmp}. When $\pi_{v}$ and $\chi_{v}$ are both ramified, the desired test vector is determined in the main local results of \cite{fmp} (Theorems 1.6, 1.7). Notice that, under the hypothesis of this paper, the place $v$ in this case corresponds to the prime ideal $\mathcal{P}$. In particular, $B_{v}$ is a matrix algebra and we have $\pi_{v}\simeq\pi'_{v}$. Assuming condition (\ref{inertcondition}), it is shown in \cite{fmp} that there exists an element $g_{v}\in GL_{2}(F_{v})$ and a subgroup $W_{v}\subset GL_{2}(O_{F_{v}})$ such that, for a non-zero $l_{v}\in \mathrm{Hom}_{K^{*}_{v}}(\pi_{v},\mathbb{C}(\chi_{v}))$, the subgroup $g_{v}W_{v}g_{v}^{-1}$ fixes a 1-dimensional subspace of $\pi_{v}$ consisting of test vectors for $l_{v}$. We let $\theta_{\chi,v}\in\pi'_{v}$ be the image under  $\pi_{v}\simeq\pi'_{v}$ of the unique (up to scalar multiples) vector fixed by this subgroup. We note that the determination of $g_{v}$ and $W_{v}$ depends on whether $K_{v}/F_{v}$ is split or inert, but either way the subgroup $g_{v}W_{v}g_{v}^{-1}$ corresponds to an order of reduced discriminant $\mathcal{N}_{v}$ which optimally contains the order of $K_{v}$ of the same conductor as $\chi_{v}$. We mention here that our choice of $k_{\mathcal{P}}$ in Lemma \ref{cmpoint} is motivated by Theorems 1.6  and 1.7 in \cite{fmp}.

Recall that $\theta$ is right invariant under $H=\hat{R}^{*}$ where for each finite place $v$ of $F$, the order $R_{v}\subset B_{v}$ is an order of reduced discriminant $\mathcal{N}_{v}$ which optimally contains the order $O_{v}$ of $K_{v}$ ($O\subset K$ is the $O_{F}$-order of $\mathcal{P}$-power conductor specified at the beginning of Section 2). Hence, there exists $b_{v}\in B^{*}_{v}$ such that $b_{v}R_{v}b_{v}^{-1}=R_{\chi,v}$. In other words, the $O_{F_{v}}$-order $K_{v}\cap b_{v}R_{v}b_{v}^{-1}$ has conductor equal to that of $\chi_{v}$ for every finite place $v$ of $F$ so that $\text{\bf{x}}=[b]\in\mathrm{CM}_{H}(\mathcal{P}^{n})$. It follows that  $\theta_{\chi}=b.\theta$ (up to scalar multiples). Since the quantity $\frac{|l_{\chi}(\theta_{\chi})|^{2}}{(\theta_{\chi},\theta_{\chi})}$ that appears in the special value formula (Theorem \ref{martin}) is invariant under scaling, we can assume without loss of generality that $\theta_{\chi}(z)=\theta(zb)$ for all $z\in G(\mathbb{A}_{f})$.

\begin{lemma}\label{finitesum}
There exists $\text{\bf{x}}\in\mathrm{CM}_{H}(\mathcal{P}^{n})$ such that $$l_{\chi}(\theta_{\chi})=\mu_{\chi}\mathrm{a}(\text{\bf{x}},\chi),$$ where $\mu_{\chi}$ is the volume of the image of $U_{\mathcal{P}^{n}}=\hat{O}_{\mathcal{P}^{n}}^{*}K_{\infty}^{*}$ in $Z(\mathbb{A})T(\mathbb{Q})\backslash T(\mathbb{A})$.
\end{lemma}
\begin{proof}
Since $\theta_{\chi}$ is right invariant under the subgroup $\hat{R}_{\chi}^{*}$ where for each finite place $v$ of $F$, the order $R_{\chi,v}\subset B_{v}$ is an order of reduced discriminant $\mathcal{N}_{v}$ which optimally contains the order of $K_{v}$ with the same conductor as $\chi_{v}$, it follows that $l_{\chi}(\theta_{\chi})$ can be written as
\begin{align*}
l_{\chi}(\theta_{\chi})&=\mu_{\chi}\sum_{Z(\mathbb{A})T(\mathbb{Q})\backslash T(\mathbb{A})/U_{\mathcal{P}^{n}}}\chi(t)\theta_{\chi}(t)\\&=\mu_{\chi}\sum_{Z(\mathbb{A})T(\mathbb{Q})\backslash T(\mathbb{A})/U_{\mathcal{P}^{n}}}\chi(t)\theta(tb)\\&=\frac{\mu_{\chi}}{|G_{2}|}\sum_{\sigma\in G(n)}\chi(\sigma)\psi(\sigma.\text{\bf{x}}),
\end{align*}
where $\text{\bf{x}}$ is the class $[b]$ in $\mathrm{CM}_{H}$. The last equality follows from that fact that  $Z(\mathbb{A})T(\mathbb{Q})\backslash T(\mathbb{A})/U_{\mathcal{P}^{n}}$ can be identified with the quotient $G(n)/G_{2}$.
\end{proof}
Notice that $\mu_{\chi}=\text{vol}(U_{\mathcal{P}^{n}},dt)=\text{vol}(U_{K},dt)L_{\mathcal{P}}(\eta,1)p^{-n}$. We also have \begin{align*}\text{vol}(U_{K},dt)&=\dfrac{\text{vol}(Z(\mathbb{A})T(\mathbb{Q})\backslash T(\mathbb{A}),dt)}{|Z(\mathbb{A})T(\mathbb{Q})\backslash T(\mathbb{A})/U_{K}|}\\&=2L(\eta,1)\frac{h_{F}}{h_{K}}\\&=\dfrac{2(2\pi)^{d}}{w_{K}Q_{K}}\sqrt{\frac{\Delta_{F}}{\Delta_{K}}},\end{align*} where $w_{K}$ is the number of roots of unity in $K$, $Q_{K}$ is the Hasse unit index of $K$ and $h_{K}/h_{F}$ is the relative class number of $K$ with respect to its maximal totally real subfield $F$. Notice that the second equality follows from the normalization of the Haar measure used in (\ref{haar}) (see \cite{fmp} and Section 2.1 in \cite{feigonwhitehouse} for more details), and the last equality follows from the well-known relative class number formula (see Chapter 4 in \cite{washington}). We now arrive at the following formula:
 \begin{equation}\label{specialvalueformula}\frac{|\mathrm{a}(\text{\bf{x}},\chi)|^{2}}{(\theta,\theta)}=\frac{p^{n}Q_{K}^{2}w_{K}^{2}}{8\pi^{d}(2\pi)^{2d}}\sqrt{\frac{\Delta_{K}}{\Delta_{F}}}L_{S(\pi)\cap S(\chi)^{c}}(\eta,1)\cdotp L_{S(\pi)\cap S(\chi)}(1_{F},1)\frac{L^{S_{2}(\pi)}(\pi,\chi,\frac{1}{2})}{L^{S_{2}(\pi)}(\pi,Ad,1)}.\end{equation}

\begin{remark}
The quantity $\frac{|\mathrm{a}(x,\chi)|^{2}}{(\theta,\theta)}$ does not depend on the choice of $\theta$ which is in our case only unique up to a unit in $E_{l}$. 
\end{remark}
\begin{remark}
 Since $\mathrm{a}(\gamma.x,\chi)=\chi^{-1}(\gamma)\mathrm{a}(x,\chi)$ for any $\gamma\in G(n)$, and $G(n)$ acts simply and transitively on $\mathrm{CM}_{H}(\mathcal{P}^{n})$,  it suffices to study the $\lambda$-adic valuation of $\mathrm{a}(y,\chi)$ for any $y\in \mathrm{CM}_{H}(\mathcal{P}^{n})$.
\end{remark}
\section{Two Level Raising Steps}
 By class field theory, we identify ring class characters of $\mathcal{P}$-power conductor with finite-order characters of $G(\infty)$. The torsion subgroup $G_{0}$ of $G_{\infty}$ is finite, and we know by Lemma $2.8$ in \cite {cor-vat} that we can identify $G_{0}$ with its image $G_{0}(n)$ in $G(n)$ whenever $n$ is sufficiently large. We then denote the quotient group $G(n)/G_{0}(n)$ by $H(n)$. We fix a character $\chi_{0}$ of $G_{0}$ such that $\chi_{0}=1$ on $G_{2}$. We consider all primitive characters $\chi$ of $G(n)$ $($do not factor through $G(n-1)$$)$ of the form $\chi=\chi'_{0}\chi_{1}$, where $\chi'_{0}$ is a character of $G(n)$ inducing $\chi_{0}$ on $G_{0}(n)\simeq G_{0}$, and $\chi_{1}$ is a primitive character of $H(n)$ which we also require to be faithful (injective). We denote the set of all such characters by $P(n,\chi_{0})$ . 

Recall that $E_{l}$ is an $l$-adically complete discrete valuation ring containing the Fourier coefficients of $g$, and $\lambda$ is the maximal ideal in $E_{l}$. Enlarge $E_{l}$ to contain the values of $\chi_{0}$ and the $p^{\text{th}}$ roots of unity. 

Given $x=[g]\in CM_{H}(\mathcal{P}^{n})$ and a character $\chi\in P(n,\chi_{0})$, we define the Gross-Zagier sum
$$\text{a}(x,\chi)=\frac{1}{|G_{2}|}\sum_{\sigma\in G(n)}\chi(\sigma)\psi(\sigma.x).$$

 Denote by $E_{\lambda}$ the residue field $E_{l}/\lambda$ and let $E_{\lambda}(\chi_{1})$ be the field obtained by adjoining to $E_{\lambda}$ the values of $\chi_{1}$.  We consider the trace of $\mathrm{a}(x,\chi)$ taken from $E_{\lambda}(\chi_{1})$ to $E_{\lambda}$: $$\mathrm{Tr}(\mathrm{a}(x,\chi))=\sum_{\sigma\in \mathrm{Gal}(E_{\lambda}(\chi_{1})/E_{\lambda})}\sigma(\mathrm{a}(x,\chi)).$$ This trace expression is different than the average expression $$\mathrm{b}(x,\chi_{0})=\sum_{\chi\in P(n,\chi_{0})}\mathrm{a}(x,\chi)$$ considered in the work of Vatsal and Cornut-Vatsal. Nevertheless, the same approach is followed to study both expressions. Evidently, the non-vanishing of $\mathrm{b}(x,\chi_{0})$ implies only the non-vanishing of $a(x,\chi)$ for some $\chi\in P(n,\chi_{0})$. However, given any $\chi\in P(n,\chi_{0})$, the non-vanishing of $\mathrm{Tr}(\mathrm{a}(x,\chi))$ would then imply the non-vanishing of $\mathrm{a}(x,\chi)$.

Notice that  \begin{align*}\mathrm{Tr}(\mathrm{a}(x,\chi))&=\frac{1}{|G_{2}|}\mathrm{Tr}_{E_{\lambda}(\chi_{1})/E_{\lambda}}\sum_{\sigma\in G_{0}(n)}\sum_{\tau\in H(n)}\chi'_{0}(\sigma)\chi_{1}(\tau)\psi(\sigma\tau.x)\\&=\frac{1}{|G_{2}|}\sum_{\sigma\in G_{0}}\chi_{0}(\sigma)\sum_{\tau\in H(n)}\psi(\sigma\tau.x)\mathrm{Tr}_{E_{\lambda}(\chi_{1})/E_{\lambda}}\chi_{1}(\tau)\\&=\frac{[E_{\lambda}(\chi_{1}):E_{\lambda}]}{|G_{2}|}\sum_{\sigma\in G_{0}}\chi_{0}(\sigma)\sum_{\substack{\tau\in H(n)\\ \chi_{1}(\tau)\in E_{\lambda}}}\psi(\sigma\tau.x)\chi_{1}(\tau).\end{align*}
Let $p^o$ be the highest power of $p$ dividing the order of $\chi_{0}$, and set $r=\max\{1,o\}$. Notice that an element $\tau$ in $H(n)$ satisfies $\chi_{1}(\tau) \in E_{\lambda}$ if and only if the order of $\tau$ divides $p^{r}$. 

For $m\ge1$ and $n\gg0$, we put $Z(n,m):=\mathrm{ker}(H(n)\twoheadrightarrow H(n-m))$. The following lemma is crucial in the sequel.
\begin{lemma}\label{Z(n,m)}
For $n\gg0$, $Z(n,m)\simeq O_{F_{\mathcal{P}}}/\mathcal{P}^{m}$ as a group.
\end{lemma}
\begin{proof}
 If $n$ is suffieciently large, the natural quotient map $G(n)\twoheadrightarrow H(n)$ induces an isomorphism between $\mathrm{ker}(H(n)\twoheadrightarrow H(n-m))$ and $\mathrm{ker}(G(n)\twoheadrightarrow G(n-m))$. We also have an isomorphism between $\mathrm{ker}(G(n)\twoheadrightarrow G(n-m))$ and $$K^{*}\hat{O}^{*}_{\mathcal{P}^{n-m}}/K^{*}\hat{O}^{*}_{\mathcal{P}^{n}}\simeq O^{*}_{\mathcal{P}^{n-m},\mathcal{P}}/O^{*}_{\mathcal{P}^{n-m}}O^{*}_{\mathcal{P}^{n},\mathcal{P}},$$ induced by the reciprocity map. Notice that $O^{*}_{\mathcal{P}^{n-m}}=O^{*}_{F}$ is contained in $O^{*}_{\mathcal{P}^{n},\mathcal{P}}$ for sufficiently large $n$, so that $$\mathrm{ker}(G(n)\twoheadrightarrow G(n-m))\simeq O^{*}_{\mathcal{P}^{n-m},\mathcal{P}}/O^{*}_{\mathcal{P}^{n},\mathcal{P}}.$$ On the other hand, $$O^{*}_{\mathcal{P}^{n-m},\mathcal{P}}/O^{*}_{\mathcal{P}^{n},\mathcal{P}}=\{1+a\alpha_{\mathcal{P}}\varpi^{n-m}_{\mathcal{P}}\mod O^{*}_{\mathcal{P}^{n},\mathcal{P}}: a\in O_{F_{\mathcal{P}}}/\mathcal{P}^{m}\},$$
thus yielding the desired isomorphism.
\end{proof}
In what follows, we set $m=er$. Notice that for $n\gg0$, we have $Z(n,m)=\{\tau\in H_{n}: \chi_{1}(\tau)\in E_{\lambda}\}$. In fact, it is only for this part of the proof that we need to assume that the inertia degree of $\mathcal{P}$ over $p$ is 1. Hence, we get
 \begin{align*}\mathrm{Tr}(\mathrm{a}(x,\chi))&=\frac{[E_{\lambda}(\chi_{1}):E_{\lambda}]}{|G_{2}|}\sum_{\sigma\in G_{0}}\chi_{0}(\sigma)\sum_{\tau\in Z(n,m)}\chi_{1}(\tau)\psi(\sigma\tau.x)\\&=\frac{[E_{\lambda}(\chi_{1}):E_{\lambda}]}{|G_{2}|}\sum_{\sigma\in G_{0}}\chi_{0}(\sigma)\sum_{a\in O_{F_{\mathcal{P}}}/\mathcal{P}^{m}}\chi_{1}(\tau_{a})\psi(\sigma\tau_{a}.x)\end{align*}

In view of the preceeding lemma, each element $\tau_{a}\in Z(n,m)$ corresponds to the class of $\lambda_{a,\mathcal{P}}=1+a\alpha_{\mathcal{P}}\varpi^{n-m}_{\mathcal{P}}$ in $O^{*}_{\mathcal{P}^{n-m},\mathcal{P}}/O^{*}_{\mathcal{P}^{n},\mathcal{P}}$ via the reciprocity map. In the following computations, we identify $\lambda_{a,\mathcal{P}}$ with its image in $GL_{2}(O_{F_{\mathcal{P}}})$:
 $$\lambda_{a,\mathcal{P}}\mapsto\left( \begin{array}{cc} 1+a\varpi^{n-m}_{\mathcal{P}}\mathrm{Tr}\alpha_{\mathcal{P}} & a\varpi^{n-m}_{\mathcal{P}}\\ -a\varpi^{n-m}_{\mathcal{P}}\mathrm{N}\alpha_{\mathcal{P}} & 1 \end{array} \right).$$
We denote by $\lambda_{a}$ the image of $\lambda_{a,\mathcal{P}}$ in $T(\mathbb{A}_{f})$. We now analyze the term $\psi(\tau_{a}.x)$ that appears in the above expression for $\mathrm{Tr}(\mathrm{a}(x,\chi))$.

Let $x=[g]$ be any CM point in $\mathrm{CM}_{H}(\mathcal{P}^{n})$. There exists $\sigma\in G(n)$  such that $x=\sigma.z$ with $z=[g_{0}]$ being the CM point specified at the end of Section 2. By definition, $\sigma.z=[\mu g_{0}]$ where $\mu\in T(\mathbb{A}_{f})$ is such that $\mathrm{rec}_{K}(\mu)=\sigma$. It follows that if $x=[g]\in\mathrm{CM}(\mathcal{P}^{n})$, then 
\begin{equation}\label{xm}
g=\eta\mu g_{0}r \text{ for some }\eta\in T(\mathbb{Q})\text{ and }r\in H.
\end{equation} 
Hence, $$\psi(\tau_{a}.x)=\theta(\lambda_{a}g)=\theta(\mu\lambda_{a}g_{0}).$$
However, notice that 
\begin{align*}
\lambda_{a,\mathcal{P}}k_{\mathcal{P}}&=k_{\mathcal{P}}\left( \begin{array}{cc} 1+a\varpi^{n-m}_{\mathcal{P}}\mathrm{Tr}\alpha_{\mathcal{P}} & a\varpi^{\delta-m}_{\mathcal{P}}\\ -a\varpi^{2n-m-\delta}_{\mathcal{P}}\mathrm{N}\alpha_{\mathcal{P}} & 1 \end{array} \right)\\&=k_{\mathcal{P}}\left( \begin{array}{cc} 1 & a\varpi^{\delta-m}_{\mathcal{P}}\\0 & 1 \end{array} \right)\left( \begin{array}{cc} 1+a\varpi^{n-m}_{\mathcal{P}}\mathrm{Tr}\alpha_{\mathcal{P}}+a^{2}\varpi^{2n-2m}_{\mathcal{P}}\mathrm{N}\alpha_{\mathcal{P}} & 0\\ -a\varpi^{2n-m-\delta}_{\mathcal{P}}\mathrm{N}\alpha_{\mathcal{P}} & 1 \end{array} \right).
\end{align*}
For every $a\in O_{F_{\mathcal{P}}}$ and $m\geq1$, we put 
\begin{equation*}
\alpha_{a,m}=(1,1,...,\underbrace{h_{\mathcal{P}}\left( \begin{array}{cc} 1 & a\varpi^{\delta-m}_{\mathcal{P}} \\ 0& 1 \end{array} \right)h_{\mathcal{P}}^{-1}}_{\mathcal{P}^{\mathrm{th}}\mathrm{place}},...,1,1)
\end{equation*} so that $\theta(\mu\lambda_{a}g_{0})=\theta(\mu g_{0}\alpha_{a,m})$. Consequently,
\begin{equation}\label{thetam}
\psi(\tau_{a}.x)=\theta(gr^{-1}\alpha_{a,m}).
\end{equation}
We now introduce a level structure $H_{m}\subset H$ which agrees with $H$ outside $\mathcal{P}$ and corresponds to an $O_{F}$-order $R_{m}\subset B$. Let us first recall the definitions of a couple of specific open and compact subgroups of $GL_{2}(F_{\mathcal{P}})$. For an ideal $\mathcal{I}$ of $O_{F_{\mathcal{P}}}$, we have $$K_{1}(\mathcal{I})=\left\{\left( \begin{array}{cc} a & b \\ c& d \end{array} \right)\in GL_{2}(O_{F_{\mathcal{P}}}):c\equiv0,a\equiv d\mod\mathcal{I}\right\},$$ and 
$$K^{1}(\mathcal{I})=\left\{\left( \begin{array}{cc} a & b \\ c& d \end{array} \right)\in GL_{2}(O_{F_{\mathcal{P}}}):b\equiv0,a\equiv d\mod\mathcal{I}\right\}.$$ 
 We put 
$$H_{m,\mathcal{P}}=R^{*}_{m,\mathcal{P}} =
\begin{cases}
R^{*}_{\mathcal{P}}\cap h_{\mathcal{P}}K_{1}(\mathcal{P}^{2m-\delta})h_{\mathcal{P}}^{-1}, &  \text{if }\delta<m\\ 
R^{*}_{\mathcal{P}}\cap h_{\mathcal{P}}K_{1}(\mathcal{P}^{m})h_{\mathcal{P}}^{-1}, & \text{if }m\leq\delta<2m\\
R^{*}_{\mathcal{P}}\cap h_{\mathcal{P}}K^{1}(\mathcal{P}^{m})h_{\mathcal{P}}^{-1}, & \text{otherwise.}
\end{cases}$$

Define a function $\theta_{m}$ on $G(\mathbb{Q})\backslash G(\mathbb{A}_{f})$ by: $$\theta_{m}=\sum_{a\in O_{F_{\mathcal{P}}}/\mathcal{P}^{m}}\chi_{1}(\tau_{a})(\alpha_{a,m}.\theta).$$ 
 \begin{lemma}
The function $\theta_{m}$ has level $H_{m}=\hat{R}_{m}^{*}$.
\end{lemma}
\begin{proof}
We need to prove that $\gamma.\theta_{m}=\theta_{m}$ for all $\gamma\in H_{m}$. Let $\gamma$ be any element in $H_{m}$, and write $$\gamma_{\mathcal{P}}=h_{\mathcal{P}}\left( \begin{array}{cc} w & x \\ y& z \end{array} \right)h_{\mathcal{P}}^{-1}.$$ In particular, we have $w\equiv z\mod \mathcal{P}^{m}$.
Notice that \begin{align*}\left( \begin{array}{cc} w & x \\ y& z \end{array} \right)\left( \begin{array}{cc} 1 & a\varpi^{\delta-m}_{\mathcal{P}} \\ 0& 1 \end{array} \right)&=\left( \begin{array}{cc} 1 &xz^{-1} \\ yw^{-1}& 1 \end{array} \right)\left( \begin{array}{cc} 1 & wz^{-1}a\varpi^{\delta-m}_{\mathcal{P}} \\ 0& 1 \end{array} \right)\left( \begin{array}{cc} w & 0 \\ 0& z \end{array} \right) \\&= \left( \begin{array}{cc} 1 & wz^{-1}a\varpi^{\delta-m}_{\mathcal{P}} \\ 0& 1 \end{array} \right)\left( \begin{array}{cc} 1-yz^{-1}a\varpi^{\delta-m}_{\mathcal{P}} & xz^{-1}+wz^{-2}ya^{2}\varpi^{2\delta-2m}_{\mathcal{P}} \\ yw^{-1}& 1+yz^{-1}a\varpi^{\delta-m}_{\mathcal{P}} \end{array}\right)\\&\times\left( \begin{array}{cc} w & 0 \\ 0& z \end{array} \right).
\end{align*}
Since $\gamma_{\mathcal{P}}\in H_{m,\mathcal{P}}$, we know that $$h_{\mathcal{P}}\left( \begin{array}{cc} 1-yz^{-1}a\varpi^{\delta-m}_{\mathcal{P}} & xz^{-1}+wz^{-2}ya^{2}\varpi^{2\delta-2m}_{\mathcal{P}} \\ yw^{-1}& 1+yz^{-1}a\varpi^{\delta-m}_{\mathcal{P}} \end{array} \right)\left( \begin{array}{cc} w & 0 \\ 0& z \end{array} \right)h_{\mathcal{P}}^{-1} \in R^{*}_{\mathcal{P}}.$$ Hence, we get 
\begin{align*}
\gamma.\theta_{m}&=\sum_{a\in O_{F_{\mathcal{P}}}/\mathcal{P}^{m}}\chi_{1}(\tau_{a})(\alpha_{wz^{-1}a,m}.\theta)\\&=\sum_{a\in O_{F_{\mathcal{P}}}/\mathcal{P}^{m}}\chi_{1}(\tau_{azw^{-1}})(\alpha_{a,m}.\theta)\\&=\sum_{a\in O_{F_{\mathcal{P}}}/\mathcal{P}^{m}}\chi_{1}(\tau_{a})(\alpha_{a,m}.\theta)=\theta_{m}.
\end{align*} 
\end{proof}
 
\begin{proposition}
Let $\psi_{m}$ be the function induced by $\theta_{m}$ on $CM_{H_{m}}$. If $n$ is sufficiently large, then for every $x\in \mathrm{CM}_{H}(\mathcal{P}^{n})$, there exists $x_{m}\in\mathrm{CM}_{H}(\mathcal{P}^{n})$ such that \begin{equation}\label{psim}
\psi_{m}(x_{m})=\sum_{\tau\in Z(n,m)}\chi_{1}(\tau)\psi(\tau.x).\end{equation}
\end{proposition}
\begin{proof}
Consider some $x=[g]\in\mathrm{CM}_{H}(\mathcal{P}^{n})$ with $g\in G(\mathbb{A}_{f})$ and $n\geq m$. Let $x_{m}$ be the class of $gr^{-1}$ in $\mathrm{CM}_{H_{m}}$, where $r$ is as specified in (\ref{xm}). It is easy to check that the CM point $x_{m}=[gr^{-1}]\in\mathrm{CM}_{H_{m}}$ has conductor $\mathcal{P}^{n}$.

To prove that formula (\ref{psim}) holds for all $x\in \mathrm{CM}_{H}(\mathcal{P}^{n})$ with $n\gg0$, we simply have to observe that
\begin{align*}
\psi_{m}(x_{m})&=\theta_{m}(gr^{-1})\\&=\sum_{a\in O_{F_{\mathcal{P}}}/\mathcal{P}^{m}}\chi_{1}(\tau_{a})\theta(gr^{-1}\alpha_{a,m})\\&=\sum_{a\in O_{F_{\mathcal{P}}}/\mathcal{P}^{m}}\chi_{1}(\tau_{a})\psi(\tau_{a}.x)\\&=\sum_{\tau\in Z(n,m)}\chi_{1}(\tau)\psi(\tau.x),
\end{align*}
 where the third equation follows trivially from (\ref{thetam}).
\end{proof}

Since $\chi_{0}=1$ on $G_{2}$, we have:  $$\mathrm{Tr}(\mathrm{a}(x,\chi))=[E_{\lambda}(\chi_{1}):E_{\lambda}]\sum_{\sigma\in G_{0}/G_{2}}\chi_{0}(\sigma)\psi_{m}(\sigma.x_{m}).$$
We can reduce the above sum into something even simpler by means of another level raising step.
\begin{proposition}
There exists an $O_{F}$-order $R_{m,\mathcal{D}}$, a non-zero function $\theta_{m,\mathcal{D}}$ of level $H_{m,\mathcal{D}}=\hat{R}_{m,\mathcal{D}}^{*}$ on $G(\mathbb{A}_{f})$, and for each $n\geq0$, a Galois equivariant map $x_{m}\mapsto x_{m,\mathcal{D}}$ from $CM_{H_{m}}(\mathcal{P}^{n})$ to $CM_{H_{m,\mathcal{D}}}(\mathcal{P}^{n})$ such that $$\psi_{m,\mathcal{D}}(x_{m,\mathcal{D}})=\sum_{\tau\in G_{1}/G_{2}}\chi_{0}(\tau)\psi_{m}(\tau.x_{m}),$$ where $\psi_{m,\mathcal{D}}=\theta_{m,\mathcal{D}}\circ\mathrm{red}$.
\end{proposition}
\begin{proof}
The reader is referred to the proof of Lemma 5.9 in \cite{cor-vat} 
\end{proof}
Hence, the trace expression simplifies to $$\mathrm{Tr}(\mathrm{a}(x,\chi))=[E_{\lambda}(\chi_{1}):E_{\lambda}]\sum_{\sigma\in G_{0}/G_{1}}\chi_{0}(\sigma)\psi_{m,\mathcal{D}}(\sigma.x_{m,\mathcal{D}}).$$

\section{The $\lambda$-adic Valuation of $\mathrm{a}(x,\chi)$}
We now study the $\lambda$-adic valuation of the sum $$\sum_{\sigma\in \mathcal{R}}\chi_{0}(\sigma)\psi_{m,\mathcal{D}}(\sigma.x_{m,\mathcal{D}}),$$ where  $\mathcal{R}$ is a set of representatives for $G_{0}/G_{1}$ containing 1.
\begin{definition}
Let $k$ be any ring. A $k$-valued function $\phi$ on $\mathrm{M}_{H}$ is said to be Eisenstein if it factors through $\mathrm{N}_{H}$ via the map $\mathrm{c}$, where as $\phi$ is said to be exceptional if there exists $z\in\mathrm{N}_{H}$ such that $\phi$ is constant on $c^{-1}(\sigma.z)$ for all $\sigma\in\mathrm{Gal}^{\mathrm{ab}}_{K}$.
\end{definition}
Choose an ideal $\mathcal{C}$ in $O_{F}$ such that $\mathrm{nrd}(H)$ contains all elements of $\hat{O}_{F}^{*}$ congruent to $1$ modulo $\mathcal{C}$. Such an integral ideal exists because $\mathrm{nrd}(H)$ is open in $\hat{F}^{*}$. For a finite prime $v$ of $F$, let $q_{v}$ be the cardinality of the residue class field at $v$. Denote by $S$ the set of all finite places of $F$ that do not divide $\mathcal{N}$ and correspond to a principal prime ideal $aO_{F}$ with $a\equiv1\mod C$ and $a$ is totally positive.
\begin{lemma}\label{eis}
Let $\phi$ be an $E_{l}$-valued function on $\mathrm{M}_{H}$ such that $\phi(x)$ is a unit for some $x\in \mathrm{M}_{H}$. If $\phi$ is Eisenstein modulo $\lambda^{r}$ for some positive integer $r$, then $a_{v}\equiv q_{v}+1\mod\lambda^{r}$ for all $v\in S$. 
\end{lemma}
\begin{proof}
Let $v$ be a finite place in $F$ corresponding to a principal prime ideal $Q=aO_{F}$ with $a\equiv 1\mod C$ and $a$ is totally positive. Choose $x=[g]\in\mathrm{M}_{H}$ such that $\phi(x)$ is a $\lambda$-adic unit. By definition, we know that $$T_{v}\phi(x)=\sum_{i\in I_{v}}\phi([g\eta_{v,i}]).$$
Here $H_{v}\left( \begin{array}{cc} a & 0 \\ 0& 1 \end{array} \right)H_{v}=\coprod_{i\in I_{v}}\eta_{v,i}H_{v}$. Notice that $c([g\eta_{v,i}])=c([g\eta_{v}])$, where $$\eta_{v}=(1,..,1,\left( \begin{array}{cc} a & 0 \\ 0& 1 \end{array} \right),1,...,1).$$
Since $\phi$ is Eisenstein modulo $\lambda^{r}$, we get 
$$T_{v}\phi(x)\equiv(1+q_{v})\phi([g\eta_{v}])\mod\lambda^{r}.$$ 
Choose $d\in G(\mathbb{Q})$ such that $\mathrm{nrd}(d)=a$. Notice that $\mathrm{nrd}(\eta_{v}^{-1}d)=(a,...,a,1,a,...,a)\equiv 1\mod C$. We thus obtain an element $h\in H$ such that $\mathrm{nrd}(h)=\mathrm{nrd}(\eta_{v}^{-1}d)$. Hence, $$\phi([g\eta_{v}])\equiv\phi([g\eta_{v}d^{-1}])\equiv\phi([gh^{-1}])\equiv\phi(x)\mod\lambda^{r}.$$ On the other hand, we know that $T_{v}\phi(x)=a_{v}\phi(x)$. Putting all of this together gives $a_{v}\phi(x)\equiv(1+q_{v})\phi(x)\mod\lambda^{r}$, which implies that $a_{v}\equiv 1+q_{v}\mod\lambda^{r}$ since $\phi(x)$ is a $\lambda$-adic unit.
\end{proof}
\begin{lemma}
If $\phi$ is exceptional but non-Eisenstein modulo $\lambda^{r}$ for some positive integer $r$, then $a_{v}$ is a $\lambda$-adic non-unit for all finite places $v$ of $F$ that are inert in $K$ and do not divide $\mathcal{N}$.
\end{lemma}
\begin{proof}
The argument given here is drawn from \cite{cor-vat}. Recall that we have an action of the group $\mathrm{Gal}_{K}^{\mathrm{ab}}$ on $\mathrm{N}_{H}$, and one can show that there are at most two $\mathrm{Gal}_{K}^{\mathrm{ab}}$-orbit in $\mathrm{N}_{H}$. If there were only one $\mathrm{Gal}_{K}^{\mathrm{ab}}$-orbit in $\mathrm{N}_{H}$, then any exceptional function on $\mathrm{M}_{H}$ would also be Eisenstein. Since $\phi$ is exceptional and non-Eisenstein modulo $\lambda^{r}$, we know there must be exactly two $\mathrm{Gal}_{K}^{\mathrm{ab}}$-orbits in $\mathrm{N}_{H}$, which we denote by $X$ and $Y$ with $\phi$ being constant modulo $\lambda^{r}$ on $\mathrm{c}^{-1}(z)$ for all $z\in X$. Since $\phi$ is non-Eisenstein modulo $\lambda^{r}$, there exist $y\in Y$ and some $x_{1},x_{2}\in\mathrm{c}^{-1}(y)$ such that $\phi(x_{1})\not\equiv\phi(x_{2})\mod\lambda^{r}$.\\
Let $v$ be a finite place of $F$ that is inert in $K$ and does not divide $\mathcal{N}$. For any $x=[g]\in\mathrm{M}_{H}$, we know that $$T_{v}\phi(x)=a_{v}\phi(x)=\sum_{i\in I_{v}}\phi([g\eta_{v,i}]).$$ We also know that if $x\in\mathrm{c}^{-1}(y)$ then $[g\eta_{v,i}]\in\mathrm{c}^{-1}(\mathrm{Frob}_{v}.y)$. Since $v$ is inert in $K$, we get $\mathrm{Frob}_{v}.y\in X$, so that $\phi$ is constant modulo $\lambda^{r}$ on $\mathrm{c}^{-1}(\mathrm{Frob}_{v}.y)$ with $\phi(v,y)$ being the common value. Hence, $$a_{v}\phi(x_{1})\equiv(1+q_{v})\phi(v,y)\equiv a_{v}\phi(x_{2})\mod\lambda^{r}.$$ It follows that $a_{v}$ is a $\lambda$-adic non-unit, since otherwise $\phi(x_{1})$ and $\phi(x_{2})$ would be congruent modulo $\lambda^{r}$.
\end{proof}
 We shall assume henceforth that $g$ satisfies the condition: $a_{v}$ is a $\lambda$-adic unit for some $v$ inert in $K$, $v\nmid \mathcal{N}$. 
 \begin{definition}\label{cst}
Let $\mu$ be the smallest integer such that $a_{v}\not\equiv 1+q_{v}\mod\lambda^{\mu}$ for some $v\in S, v\nmid \mathcal{D}$. 
\end{definition}
It follows immediately from the definition of $\mu$ that the function $\theta$ as specified in Section 5 is non-exceptional modulo $\lambda^{\mu}$. 
 
If $\lambda$ lies above $p$ $(l=p)$, we let $s$ be the corresponding ramification index. In this case, we denote by $k$ the ring $E_{l}/\lambda^{sm+\mu}E_{l}$. If $\lambda$ does not lie above $p$ $(l\neq p)$, we denote by $k$ the ring $E_{l}/\lambda^{\mu}E_{l}$. We shall view $\theta$, $\theta_{m}$ and $\theta_{m,\mathcal{D}}$ as  $k$-valued functions. 
\begin{proposition}\label{nonzero}
The function $\theta_{m}:M_{H_{m}}\rightarrow k$ is a non-zero eigenfunction for all Hecke operators $T_{v}$ $(v\nmid \mathcal{P}\mathcal{N}'\mathcal{D}')$ with $T_{v}\theta_{m}=a_{v}\theta_{m}$. 
\end{proposition}
\begin{proof}
It is clear that $\theta_{m}$ is an eigenfunction for all Hecke operators $T_{v}$ $(v\nmid \mathcal{P}\mathcal{N}'\mathcal{D}')$ with $T_{v}\theta_{m}=a_{v}\theta_{m}$.

For every $u\in(O_{F_{\mathcal{P}}}/\mathcal{P}^{m})^{*}$, we put
\begin{equation*}
\beta_{u}=(1,1,...,\underbrace{h_{\mathcal{P}}\left( \begin{array}{cc} u &0 \\ 0& 1 \end{array} \right)h_{\mathcal{P}}^{-1}}_{\mathcal{P}^{\mathrm{th}}\mathrm{place}},...,1,1),
\end{equation*}
where the element $h_{\mathcal{P}}$ is as specified in (\ref{order}).
If $\theta_{m}=0$ as a $k$-valued function, then  
\begin{align*}
0&=\sum_{u\in(O_{F_{\mathcal{P}}}/\mathcal{P}^{m})^{*}}\beta_{u}.\theta_{m}\\&=\sum_{u\in(O_{F_{\mathcal{P}}}/\mathcal{P}^{m})^{*}}\sum_{a\in O_{F_{\mathcal{P}}}/\mathcal{P}^{m}}\chi_{1}(\tau_{ua})\alpha_{a,m}.\theta \\&=\sum_{a\in O_{F_{\mathcal{P}}}/\mathcal{P}^{m}}\alpha_{a,m}.\theta\sum_{u\in(O_{F_{\mathcal{P}}}/\mathcal{P}^{m})^{*}}\chi_{1}(\tau_{ua})
\end{align*} 
 By means of Lemma \ref{char} below, we get  \begin{align*}0&=p^{m-1}(p-1)\theta-p^{m-1}\sum_{a\in{(O_{F_{\mathcal{P}}}/\mathcal{P}})^{*}}\alpha_{a,1}.\theta\\&=p^{m}\theta-p^{m-1}\sum_{a\in O_{F_{\mathcal{P}}}/\mathcal{P}}\alpha_{a,1}.\theta\\&=p^{m-1}\left(p\theta-\sum_{a\in O_{F_{\mathcal{P}}}/\mathcal{P}}\alpha_{a,1}.\theta\right)\\&=p^{m-1}\theta^{+}.\end{align*}This yields a contradiction since $p^{m-1}\theta^{+}$ is non-zero by Lemma 4.12 in \cite{cor-vat}; the proof of this lemma uses the fact that $\theta$ is non-eisenstein modulo $\lambda^{\mu}$. The reader is referred to \cite{cor-vat} for a description of the function $\theta^{+}$ and its properties $($see, for example, Section 1.6, Theorem 5.10 and the Appendix$)$.
\end{proof}
\begin{corollary}
$\theta_{m}$ is non-exceptional as a $k$-valued function.
\end{corollary}
\begin{lemma}\label{char}
For $a\in O_{F_{\mathcal{P}}}/\mathcal{P}^{m}$, we have
 \begin{align*}\sum_{u\in(O_{F_{\mathcal{P}}}/\mathcal{P}^{m})^{*}}\chi_{1}(\tau_{ua})&=\left\{
        \begin{array}{ll}
            p^{m-1}(p-1) & \quad a\in \mathcal{P}^{m} \\
            -p^{m-1} & \quad a\in \mathcal{P}^{m-1}/\mathcal{P}^{m} \text{ and } a\notin \mathcal{P}^{m}\\
            0 & \quad \text{otherwise} 
        \end{array}
    \right.\end{align*}
\end{lemma}
\begin{proof}
The statement of the lemma follows trivially for $a\equiv0\mod\mathcal{P}^{m}$. For the remaining cases, we write $$\sum_{u\in(O_{F_{\mathcal{P}}}/\mathcal{P}^{m})^{*}}\chi_{1}(\tau_{ua})=\sum_{u\in O_{F_{\mathcal{P}}}/\mathcal{P}^{m}}\chi_{1}(\tau_{ua})-\sum_{u\in\mathcal{P}/\mathcal{P}^{m}}\chi_{1}(\tau_{ua})$$ Notice that if $a\in \mathcal{P}^{m-1}/\mathcal{P}^{m}$, we have $$\sum_{u\in O_{F_{\mathcal{P}}}/\mathcal{P}^{m}}\chi_{1}(\tau_{ua})=0 \text{ and }\sum_{u\in\mathcal{P}/\mathcal{P}^{m}}\chi_{1}(\tau_{ua})=p^{m-1}.$$Otherwise, we get $$\sum_{u\in O_{F_{\mathcal{P}}}/\mathcal{P}^{m}}\chi_{1}(\tau_{ua})=\sum_{u\in\mathcal{P}/\mathcal{P}^{m}}\chi_{1}(\tau_{ua})=0$$. \end{proof}
\begin{proposition}\label{propD}
The function $\theta_{m,\mathcal{D}}:M_{H_{m,\mathcal{D}}}\rightarrow k$ is a non-zero eigenfunction for all Hecke operators $T_{v}$ away from $\mathcal{P}\mathcal{N}'\mathcal{D}'$ with $T_{v}\theta_{m,\mathcal{D}}=a_{v}\theta_{m,\mathcal{D}}$.
\end{proposition}
\begin{proof}
By definition $($see \cite{cor-vat} p. $57$$)$, $$\theta_{m,\mathcal{D}}=\sum_{d\mid \mathcal{D}'}\chi_{0}(\sigma_{d})(\alpha_{d}.\theta_{m}),$$ where $\displaystyle{\alpha_{d}=\prod_{Q\mid d}\alpha_{Q}}$, and $\alpha_{Q}$ is an element in $R_{Q}\sim M_{2}(O_{F_{Q}})$ whose reduced norm is a uniformizer in $O_{F_{Q}}$. Notice that $\theta_{m,\mathcal{D}}$ is left-invariant under $H_{m,\mathcal{D}}=\hat{R}_{m,\mathcal{D}}$ where $R_{m,\mathcal{D}}$ is the unique $O_{F}$-order which agrees with $R_{m}$ outside $\mathcal{D}'$ and equals $R_{Q}\cap\alpha_{Q} R_{Q}\alpha_{Q}^{-1}$ at $Q\mid \mathcal{D}'$. 

If we fix  a prime divisor $Q$ of $\mathcal{D}'$, it is easy to see that $\theta_{m,\mathcal{D}}$ can be rewritten as $$\theta_{m,\mathcal{D}}=\sum_{d\mid\frac{\mathcal{D}'}{Q}}\chi_{0}(\sigma_{d})(\alpha_{d}.\theta_{m})+\chi_{0}(\sigma_{Q})\alpha_{Q}.\sum_{d\mid\frac{\mathcal{D}'}{Q}}\chi_{0}(\sigma_{d})(\alpha_{d}.\theta_{m}).$$ Let $\vartheta_{1}$ and $\vartheta_{2}$ be $k$-valued functions on $M_{H_{m}}$ satisfying $T_{v}\vartheta_{i}=a_{v}\vartheta_{i}$ for all $v\nmid \mathcal{P}\mathcal{N}'\mathcal{D}'$. We claim that any nontrivial linear combination $a\vartheta_{1}+b\alpha_{Q}.\vartheta_{2}$ is non-zero in $k$. If $a\vartheta_{1}+b\alpha_{Q}.\vartheta_{2}=0$ for some scalars $a$ and $b$, then $a\vartheta_{1}=-b\alpha_{q}.\vartheta_{2}$ is fixed under the group spanned by $R_{Q}^{*}$ and $\alpha_{Q}R_{Q}^{*}\alpha_{Q}^{-1}$ which contains the kernel of the reduced norm map $B_{\mathcal{P}}\rightarrow F_{\mathcal{P}}$. It follows from the strong approximation theorem $($\cite{Vig} p. $81$$)$ that $\vartheta_{1}$ factors through the norm map as a $k$-valued function, which is a contradiction to the fact that $a_{v}\not\equiv q_{v}+1\mod \lambda^{\mu}$ for some $v\in S$ $($Lemma \ref{eis}$)$. Hence, $a\vartheta_{1}+b\alpha_{Q}.\vartheta_{2}$ is non-zero. Not only this, but  $a\vartheta_{1}+b\alpha_{Q}.\vartheta_{2}$ is also an eigenfunction for all $T_{v}$ $(v\nmid \mathcal{P}\mathcal{N}'\mathcal{D}')$ with the same eigenvalues as $\vartheta_{1}$ and $\vartheta_{2}$.

In light of the above observation, we proceed by induction on the number of prime ideal divisors of $\mathcal{D}'$ to prove that $\theta_{m,\mathcal{D}}$ is non-zero and satisfies $T_{v}\theta_{m,\mathcal{D}}=a_{v}\theta_{m,\mathcal{D}}$ for all $v\nmid \mathcal{P}\mathcal{N}'\mathcal{D}'$. This reduces the problem to the case of $\theta_{m}$ which satisfies the required hypothesis by Proposition \ref{nonzero}.
\end{proof}
\begin{corollary}
$\theta_{m,\mathcal{D}}$ is non-exceptional as a $k$-valued function.
\end{corollary}
Now we state and prove the main result in this paper. This result gives an upper bound for the $l$-adic valuation of  the sum $$\sum_{\tau\in \mathcal{R}}\chi_{0}(\tau)\psi_{m,\mathcal{D}}(\tau.x_{m,\mathcal{D}}),$$ which we recall is related to the Gross-Zagier sum $\mathrm{a}(x,\chi)$ by the formula $$\mathrm{Tr}(\mathrm{a}(x,\chi))=[E_{\lambda}(\chi_{1}):E_{\lambda}]\sum_{\tau\in \mathcal{R}}\chi_{0}(\tau)\psi_{m,\mathcal{D}}(\tau.x_{m,\mathcal{D}}).$$
\begin{theorem}\label{mainthm}
Let $\chi_{0}$ be any character of $G_{0}$. For any $x\in CM_{H_{m,\mathcal{D}}}(\mathcal{P}^{n})$ with $n\gg0$, there exists some $y\in G(\infty).x$ such that 
$$\sum_{\tau\in \mathcal{R}}\chi_{0}(\tau)\psi_{m,\mathcal{D}}(\tau.y)\not\equiv 0\quad(\text{in } k).$$
\end{theorem}
\begin{proof}
We follow the proof of Corollary $5.7$ in \cite{cor-vat}. Since $\theta_{m,\mathcal{D}}$ is non-exceptional as a $k$-valued function, there exists $\sigma\in G(\infty)$ such that $\theta_{m,\mathcal{D}}$ is non-constant as a $k$-valued function on $\mathrm{c}^{-1}(\mathrm{c}\circ\mathrm{red}(\sigma.x))$. Choose $p_{1},p_{2}\in\mathrm{c}^{-1}(\mathrm{c}\circ\mathrm{red}(\sigma.x))$ such that $\theta_{m,\mathcal{D}}(p_{1})\not\equiv\theta_{m,\mathcal{D}}(p_{2})$ $($in $k)$. If $n$ is sufficiently large, Theorem \ref{cv} guarantees the existence of $y_{1},y_{2}\in G(\infty).x$ such that $$\mathrm{red}(y_{1})=p_{1},\text{ }\text{ }\text{ }\mathrm{red}(y_{2})=p_{2}$$ and $$\mathrm{red}(\tau.y_{1})=\mathrm{red}(\tau.x)=\mathrm{red}(\tau.y_{2})\text{ }\text{ }\text{ }\text{ for all } \tau\neq1\text{ in }\mathcal{R}.$$ We thus obtain \begin{align*}\sum_{\tau\in \mathcal{R}}\chi_{0}(\tau)\psi_{m,\mathcal{D}}(\tau.y_{1})-\sum_{\tau\in \mathcal{R}}\chi_{0}(\tau)\psi_{m,\mathcal{D}}(\tau.y_{2})&=\theta_{m,\mathcal{D}}(p_{1})-\theta_{m,\mathcal{D}}(p_{2})\\&\not\equiv0\text{ }(\text{in }k).\end{align*} Therefore, at least one of the sums $\sum_{\tau\in \mathcal{R}}\chi_{0}(\tau)\psi_{m,\mathcal{D}}(\tau.y_{1})$ or $\sum_{\tau\in \mathcal{R}}\chi_{0}(\tau)\psi_{m,\mathcal{D}}(\tau.y_{2})$ is non-zero in $k$.
\end{proof}

 We remark at the end of this section that one can easily obtain a lower bound on the $l$-adic valuation of the Gross-Zagier sum $\mathrm{a}(x,\chi)$. In fact, let $\nu$ be the largest integer such that $\theta$ is Eisenstein modulo $\lambda^{\nu}$. Then \begin{align*}\sum_{\sigma\in G(n)}\chi(\sigma)\theta\circ\text{red}(\sigma.x)&=\sum_{\sigma\in G(n)}\chi(\sigma)\theta(\text{red}(\sigma.x))\\&\equiv\sum_{\sigma\in G(n)}\chi(\sigma)\theta(\text{c}\circ\text{red}(\sigma.x))\\&\equiv\sum_{\sigma\in G(n)}\chi(\sigma)\theta(\sigma.\text{c}\circ\text{red}(x))\\&\equiv\sum_{\sigma\in G(n)}\chi(\sigma)\theta(\text{nrd}(\beta)\text{c}\circ\text{red}(x))\\&\equiv0\mod\lambda^{\nu},\end{align*} where the last line follows from the orthogonality property of group characters. Hence, $$ord_{\lambda}\left(\sum_{\sigma\in G(n)}\chi(\sigma)\psi(\sigma.x)\right)\geq\nu.$$ It is obvious that this simple observation combined with Theorem \ref{mainthm} would give an exact value for the $l$-adic valuation of $\mathrm{a}(x,\chi)$ if we have $\nu+1=\mu$. However, it is not clear to us whether the statement $\nu+1=\mu$ is true or not. This is a very interesting question, but we choose not to discuss it in this work. We remark only that the answer seems to be connected to multiplicity-one-type results for the component group of a Shimura curve at Eisenstein primes.

\section{Toward Computing $ord_{\lambda}(L^{\mathrm{al}}(\pi,\chi,\frac{1}{2}))$}

In this final section, we establish an upper bound on the $\lambda$-adic valuation of the algebraic part of the special value $L(\pi,\chi,\frac{1}{2})$. Since the special value formula in Theorem \ref{martin} is given in terms of $\dfrac{L^{S_{2}(\pi)}(\pi,\chi,\frac{1}{2})}{L^{S_{2}(\pi)}(\pi,Ad,1)}$, we assume for simplicity of exposition that $\mathcal{N}$ is squarefree so that $S_{2}(\pi)=\emptyset$. We also assume that $l\neq p$ so that Theorem \ref{mainthm} simply implies that $\mathrm{ord}_{\lambda}(a(x,\chi))\leq\mu-1$, where $\mu$ is the constant defined in Definition \ref{cst}. 

Denote by $\mathbf{T}_{0}$ the Hecke algebra over $E_{l}$ formed in the usual way from adelic cuspidal Hilbert modular forms of parallel weight $(2,\cdotp\cdotp\cdotp,2)$ level $\mathcal{N}$ and trivial central character. Let $\pi_{g}:\mathbf{T}_{0}\rightarrow E_{l}$ be the canonical $E_{l}$-algebra homomorphism associated to the Hilbert modular form $g$, and let $\eta_{0}$ be an $E_{l}$-generator of the $E_{l}$-ideal $\pi_{g}(\mathrm{Ann}_{\mathbf{T}_{0}}(\mathrm{ker}(\pi_{g})))$. We now define the algebraic part of $L(\pi,\chi,\frac{1}{2})$ as $$L^{\mathrm{al}}(\pi,\chi,\frac{1}{2})=\frac{L(\pi,\chi,\frac{1}{2})}{\Omega_{g}^{\mathrm{can}}},$$ where $\Omega_{g}^{\text{can}}=4(4\pi)^{d}(2\pi)^{2d}\frac{(g,g)}{\eta_{0}}$ and $(g,g)$ is the Petersson norm as defined in \cite{hida}. The period $\Omega_{g}^{\mathrm{can}}$ is analogous to Hida's canonical period referred to in \cite{vatsal2} and \cite{pw}. 

We know that $L(\pi, Ad,1)=\dfrac{2^{2d-1}(g,g)}{\Delta_{F}^{2}h_{F}|\mathcal{N}|}$ (see \cite{fmp} \S 8.5). Using this identity, we can rewrite the special value formula (\ref{specialvalueformula}) as: $$|\mathrm{a}(\text{\bf{x}},\chi)|^{2}=\frac{L(\pi,\chi,\frac{1}{2})}{\Omega}\cdotp C_{\chi}(\pi),$$ where $$\Omega=4(4\pi)^{d}(2\pi)^{2d}\dfrac{(g,g)}{(\theta,\theta)}$$ and $$C_{\chi}(\pi)=p^{n}Q_{K}^{2}w_{K}^{2}h_{F}|\mathcal{N}|\Delta_{F}^{\frac{3}{2}}\Delta_{K}^{\frac{1}{2}}\cdotp L_{S(\pi)\cap S(\chi)^{c}}(\eta,1)L_{S(\pi)\cap S(\chi)}(1_{F},1).$$

Hence, we have $$|\mathrm{a}(\text{\bf{x}},\chi)|^{2}=\frac{L(\pi,\chi,\frac{1}{2})}{\Omega_{g}^{\mathrm{can}}\frac{\eta_{0}}{(\theta,\theta)}}\cdotp C_{\chi}(\pi).$$
Notice that if we don't want to assume that $S_{2}(\pi)=\emptyset$, all we have to do is multiply $C_{\chi}(\pi)$ by $\frac{L_{S_{2}(\pi)}(\pi,Ad,1)}{L_{S_{2}(\pi)}(\pi,\chi,\frac{1}{2})}.$ 

Adapting the notation from \cite{vatsal2}, we denote the ratio $\dfrac{\eta_{0}}{(\theta,\theta)}$ by $C_{\mathrm{csp}}$. We remark here that the denominator $(\theta,\theta)$ is mistakenly replaced in \cite{vatsal2} Lemma 2.5 by some congruence number $\eta_{B}$. The two quantities are not equal in general. This issue is also mentioned and addressed to some extent  in \cite{pw} Remark 2.4 and Theorem 6.2.
\begin{lemma}
Let $\chi_{0}$ be a given character of $G_{0}$ which is trivial on $G_{2}$. We have $$\mathrm{ord}_{\lambda}\left(|\mathrm{a}(\text{\bf{x}},\chi)|^{2}\right)\leq 2(\mu-1)$$ for all $\chi\in P(n,\chi_{0})$ with $n\gg0$.
\end{lemma}
\begin{proof}
We know that $\mathrm{ord}_{\lambda}\left(\mathrm{a}(\text{\bf{x}},\chi)\right)\leq \mu-1$ from Theorem \ref{mainthm}. Notice that $|\mathrm{a}(\text{\bf{x}},\chi)|^{2}=\mathrm{a}(\text{\bf{x}},\chi)\mathrm{a}(\text{\bf{x}},\chi^{-1})$. Applying Theorem \ref{mainthm} with $\chi_{0}^{-1}$ instead of $\chi_{0}$ gives $\mathrm{ord}_{\lambda}\left(\mathrm{a}(\text{\bf{x}},\chi^{-1})\right)\leq \mu-1$ since $\chi^{-1}\in P(n,\chi_{0}^{-1})$. Therefore, $\mathrm{ord}_{\lambda}\left(|\mathrm{a}(\text{\bf{x}},\chi)|^{2}\right)\leq 2(\mu-1)$ for all $\chi\in P(n,\chi_{0})$ whenever $n$ is sufficiently large.
\end{proof}

\begin{corollary}
Let $\chi_{0}$ be a given character of $G_{0}$ which is trivial on $G_{2}$. We have $$\mathrm{ord}_{\lambda}\left(\frac{L^{\mathrm{al}}(\pi,\chi,\frac{1}{2})}{C_{\mathrm{csp}}}\cdotp C_{\chi}(\pi)\right)\leq 2(\mu-1)$$ for all $\chi\in P(n,\chi_{0})$ with $n\gg0$.
\end{corollary}
\section*{Acknowledgements}
It is a pleasure to thank Nike Vatsal for his guidance and many helpful discussions on the subject of this paper. I am also grateful to an anonymous referee for a number of corrections and suggestions that improved the exposition of this work.
\bibliographystyle{siam}
\bibliography{myrefs}

\end{document}